\documentclass[10pt, a4paper]{article}
\usepackage[english]{babel}
\usepackage{amsmath}
\usepackage{cmll}
\usepackage{amssymb}
\usepackage{amsthm}
\usepackage{bbm}
\usepackage{mathrsfs}
\usepackage{stmaryrd}
\usepackage{comandi}
\usepackage[all]{xy}
\usepackage{bussproofs}

\newdir{|>}{%
!/4.5pt/@{|}*:(1,-.2)@^{>}*:(1,+.2)@_{>}}

\mathcode`\:="603A     
\mathchardef\colon="303A  

\mathcode`\<="4268     
\mathcode`\>="5269     
\mathchardef\gt="313E  
\mathchardef\lt="313C  

\theoremstyle{definition}
\newtheorem{deff}{Definition}[section]
\newtheorem{prop}[deff]{Proposition}

\newtheorem{lem}[deff]{Lemma}

\newtheorem{Remark}[deff]{Remark}

\date{}
\begin{document}
\title{Remarks on the Tripos To Topos Construction: extensionality, comprehensions, quotients and cauchy-complete objects}
\author{Fabio Pasquali\footnote{Part of the research in this paper was carried out while the author worked at Utrecht University in the NWO-Project `The Model Theory of Constructive Proofs' nr.  613.001.007.}}\maketitle

\begin{abstract}
We give a description of the tripos-to-topos constructions in terms of four free constructions. We prove that these compose up to give a free construction from the category of triposes and logical morphisms to the category of toposes and logical functors. Then we show that other similar constructions, i.e. the one given by Frey in \cite{frey} and that of Carboni in \cite{carbons} are instances of this one.
\end{abstract}

\section*{Introduction}\label{sec0}
One of the most relevant feature one meets in dealing with the theory of triposes is the Tripos To Topos construction. The Martin Hyland effective topos and the topos of sheaves over a locale, although they have quite different properties, e.g the former is not Grothendieck, they both are relevant instances of that construction \cite{tripinret}. The construction was presented in \cite{Tripos}, where triposes were suitable Sets-indexed collections of Heyting algebras, and Pitts, in his PhD thesis, generalized the theory to triposes over an arbitrary base \cite{PittsPhd}.\\
Loosely speaking, the Tripos To Topos construction can be seen as a way to add to a tripos exactely those properties that one needs to obtain a subobjects tripos of an elementary topos. Those properties are: having comprehensions (in the sense of Lawvere),  quotients, an extensional equality and cauchy-complete base (definitions are given in the notes). It turns out that each one of those properties can be freely add to a tripos, providing (by composition) an adjoint situation between $\mathcal{L}$\textbf{Tripos}, the category of triposes and logical functors, and its subcategory on triposes with comprehensions, quotients, extensional equality and a cauchy-complete base. Moreover the latter is equivalent to $\mathcal{L}$\textbf{Topos}, the category of elementary toposes and logical functors. In other words there exists an adjoint situation
\[
\xymatrix{
\mathcal{L}\textbf{Tripos}\ar@/^/[rr]&&\mathcal{L}\textbf{Topos}\ar@/^/[ll]\ar@{}[ll]|{\bot}
}
\]
in which the right adjoint is the obvious forgetful functor, while the reflector is the functor that freely adds the four properties listed above.
\section{Regular doctrines and Triposes}\label{sec1}
This section contains the basic definitions regarding doctrines and triposes and categories of these, see \cite{tripinret} for details.\\

Let \textbf{ISL} be the category of inf-semilattices and homomorphisms.

\begin{deff}\label{def1}
A doctrine is a pair $(\C, P)$ where $\C$ is a non-empty category with binary products and $P$ a functor$$P:\C^{op}\arr \textbf{ISL}$$
\end{deff}
We shall refer to $\C$ as the base category of the doctrine. We will often write $f^*$ instead of $P(f)$ to denote the action of the functor $P$ on the morphism $f$ of $\C$ and we shall call it reindexing along $f$. Binary meets in inf-semilattices are denoted by $\wedge$. Elements in $P(A)$ will often be called formulas over $A$ and the top element is denoted by $\top_A$.
\begin{deff}\label{def2}
A doctrine is \textbf{regular} if for each arrow $f : X \arr Y$ in $\C$ there exists a functor $\exists_{f} \dashv f^*$ satisfying
\begin{itemize}
\item[-]Beck-Chevalley condition: i.e. for every pullback of the form 
\[\xymatrix{
X\ar[d]_-{g}\ar[r]^-{f}&Y\ar[d]^{h}\\
Z\ar[r]_-{k}&W
}\]
it holds that $\exists_f\circ g^*= h^*\circ\exists_k$
\item[-]Frobenius Reciprocity: i.e. $\exists_f (\alpha\wedge f^*\beta) = \exists_f \alpha \wedge \beta$.
\end{itemize}
\end{deff}
We denote by \textbf{RD} the category of regular doctrines and regular functors:
the \textbf{objects} are regular doctrines and
the \textbf{arrows} are pairs $(F,f)$
$$
\xymatrix@R=2.0ex@C=2.5ex{
\C^{op}\ar[rrd]^-{P}="a"\ar@/_1pc/ [dd]_{F^{op}}="c"&&\\
&&\textbf{ISL}\\
\mathbb{D}^{op}\ar[rru]_-{R}="b"&&
\ar@/_1pc/"a";"b"_{f}}
$$
where the functor $F:\C\arr\D$ preserves binary products and $f$ is a natural transformation from the functor $P$ to the functor $R\circ F$ which commutes with left adjoints, i.e. it preserves the regular structures ($\exists$, $\wedge$, $\top$ and $\delta$).
\begin{deff}
A regular doctrine $(\C,P)$ is a \textbf{tripos} if
\begin{itemize}
\item[-] $P$ factors through the category of Heyting algebras and homomorphisms
\item[-] for every morphism $f$ in $\C$ the functor $f^*$ has a right adjoint $\forall_f$ satisfying Beck-Chevalley condition
\item[-] for every $X$ in $\C$, there exists $\mathbb{P}X$ in $\C$ and a formula $\in_X$ in $P(X\times \mathbb{P}X)$ such that for every object $Y$ in $\C$ and formula $\gamma$ in $P(X\times Y)$ there exists $\{\gamma\}:Y\arr \mathbb{P}X$ such that $(id_X\times \{\gamma\})^*\in_X = \gamma$.
\end{itemize}
\end{deff}
We shall abbreviate with $\delta_A$ the formula $\exists_{<id_A,id_A>}\top_A$ of $P(A\times A)$ and call it internal equality predicate over $A$. We say that $\mathbb{P}X$ is a \textbf{weak} power object of $X$, moreover we say that a tripos has \textbf{strong} power objects if, for every $A$ in $\C$, there exists an  object $\mathbb{P}A$ in $\C$ such that $\mathbb{P}A$ is a weak power object of $A$ and $$\delta_{\mathbb{P}A} = \forall_{<\pi_2,\pi_3>}(<\pi_1,\pi_2>^*\in_A\sse<\pi_1,\pi_3>^*\in_A)$$
We call \textbf{LT} the subcategory of \textbf{RD} of triposes and logical functors: the objects are triposes and the arrows are those arrows $(F,f)$ of \textbf{RD} such that $f$ is an homomorphism of Heyting algebras which commutes with left and right adjoints, i.e. it preserves all the first order structure ($\forall$, $\exists$, $\imply$, $\wedge$, $\lor$, $\top$, $\bot$ and $\delta$) and $F$ preserves weak power objects, i.e. $F$ maps a weak power object of $A$ into a weak power object of $FA$.
\section{Comprehensions}\label{sec2}
We recall in this section some known constructions involving the notion of comprehensions and related properties. 
\begin{deff}
A doctrine $(\C,P)$ is said to have \textbf{comprehensions} if for every object $A$ of $\C$ and formula $\alpha$ over $A$, there exists a moprhism $\lfloor \alpha \rfloor:X\arr A$ such that $\lfloor \alpha \rfloor^*\alpha = \top_X$ and for every $f:Y\arr A$ with $f^*\alpha = \top_X$ there exists a unique morphism $h:Y\arr X$ with $\lfloor \alpha \rfloor \circ h = f$. Moreover $\lfloor \alpha \rfloor$ is said to be \textbf{full} if for every formula $\beta$ over $A$, $\lfloor \alpha \rfloor^*\alpha \le \lfloor \alpha \rfloor^*\beta$ if and only if $\alpha\le\beta$.
\end{deff}
We denote by \textbf{RD}$_{(c )}$ the subcategory of \textbf{RD} whose objects are regular doctrines with full comprehensions and morphisms are pairs $(F,f)$ of \textbf{RD} in which $F$ preserves comprehensions. The inclusion of \textbf{RD}$_{(c )}$ in \textbf{RD} has a left adjoint $\textbf{c}:\textbf{RD}\arr\textbf{RD}_{(c )}$.\\\\ We briefly recall it from \cite{RM}. Given a regular doctrine $(\C,P)$, we denote the free regular doctrine with full comprehensions by $(\C_c,P_c)$. $\C_c$ is the category whose objects are pairs $(A,\alpha)$ in which $\alpha$ is an element of $P(A)$, while a morphism $f: (A,\alpha)\arr (B,\beta)$ is a morphism $f:A\arr B$ in $\C$ such that $\alpha \le f^*\beta$. Identity and composition are those of $\C$.\\\\ The functor $P_c$ is defined by the following assignment on objects of $\C_c$ $$P_c(A,\alpha) = \{\phi\ \epsilon \ P(A)\ |\ \phi \le \alpha\}$$ and by the following assignment on morphisms $f:(A,\alpha) \arr (B,\beta)$ of $\C_c$ $$P_c(f)(\psi) = f^*\psi \wedge \alpha$$
It is straightforward to see that $P_c(A,\alpha)$ is an infsemilattice: if $\phi\le \alpha$ and $\psi\le\alpha$, then $\phi \wedge \psi\le\alpha$, while the top element is $\alpha$ itself.\\\\
The unite of the adjunction is the family of morphisms $(H,\eta)_P:(\C,P)\arr(\C_c,P_c)$ where $H:\C\arr\C_c$ sends every object $A$ to $(A,\top_A)$ and acts as the identity on morphisms, while $\eta$ is the family of identity homomorphisms between inf-semilattices, since $P_c(A,\top_A) = P(A)$.\\

The left adjoint restricts to the inclusion of \textbf{LT}$_{(c )}$ into \textbf{LT} where the former is category of triposes with full comprehensions and logical morphisms preserving them. It is straightforward to see that for $\phi$ a formula over $(A,\alpha)$, the functor $\psi \mapsto (\phi\imply\psi) \wedge \alpha$ from $P_c(A,\alpha)$ to itself is right adjoint to $\psi\mapsto \psi \wedge \phi$. Moreover given $f:(A,\alpha)\arr(B,\beta)$, the right adjoint to $P_c(f)$ is $\forall_f (\alpha \imply -) \wedge \beta$.\\\\ Finally the following assignments $$\mathbb{P}(A,\alpha) = (\mathbb{P}(A), \forall_{\pi_2}(\in_A\imply \pi_1^*\alpha))$$$$\in_{(A,\alpha)} = \in_A  \wedge\ \pi_2^*\forall_{\pi_2}(\in_A\imply \pi_1^*\alpha))$$ determine a weak power object of $(A,\alpha)$, making $(\C_c,P_c)$ a tripos.\\\\ It is immediate to see that $(H,\eta)_P:P\arr P_c$ is a logical morphism by replacing, in the assignments above, $\alpha$ and $\beta$ with $\top_A$ and $\top_B$ respectively. Moreover if $(F,f):(\C,P)\arr (\D,R)$ is logical, its unique extension $(\overline{F}, \overline{f}): (\C_c,P_c)\arr(\D,R)$ is also logical. Then the restriction of the functor \textbf{c} to \textbf{LT}, which we denote also by \textbf{c}, fits in the following commutative diagram 
\[
\xymatrix{
\textbf{RD}\ar@/^/[r]^-{\textbf{c}}&\textbf{RD}_{(c )}\ar@/^/[l]\ar@{}[l]|{\bot}\\
\textbf{LT}\ar@{^(->}[]+<0ex,2.5ex>;[u]\ar@/^/[r]^-{\textbf{c}}&\textbf{LT}_{(c )}\ar@{^(->}[]+<0ex,2.5ex>;[u]\ar@/^/[l]\ar@{}[l]|{\bot}
}
\]
\section{Quotients}\label{sec3}
Before recalling from \cite{RM2} the definition of quotients, we need the notion of equivalence relation in a regular doctrine.\\\\
In a given regular doctrine $(\C,P)$, an equivalence relation $\rho$ over an object $A$ of $\C$ is a formula over $A\times A$ such that the following conditions hold
\begin{itemize}
\item[]$\delta_A\le \rho$
\item[]$\rho = <\pi_2,\pi_1>^*\rho$
\item[]$<\pi_1,\pi_2>^*\rho\wedge<\pi_2,\pi_3>^*\rho\le<\pi_1,\pi_3>^*\rho$
\end{itemize}
\begin{deff}
A regular doctrine $(\C,P)$ is said to have \textbf{quotients} if for every $A$ in $\C$ and every equivalence relation $\rho$ over $A$ there exists a morphism $q:A\arr A/\rho$ such that $\rho \le (q\times q)^*\delta_{A/\rho}$ and for every morphism $f:A\arr Y$ such that $\rho \le (f\times f)^*\delta_Y$ there exists a unique $h:A/\rho\arr Y$ with $h\circ q = f$. $A/\rho$ is said to be an \textbf{effective} quotients if $\rho = (q\times q)^*\delta_{A/\rho}$. 
\end{deff}
We denote by \textbf{RD}$_{(c,q)}$ the subcategory of \textbf{RD}$_{(c )}$ whose object are regular doctrines with full comprehensions and effective quotients and morphisms are those morphisms of \textbf{RD}$_{(c )}$ which preserve quotients.\\\\
Maietti and Rosolini in \cite{RM2} proved that the inclusion of  \textbf{RD}$_{(c,q)}$ in \textbf{RD}$_{(c )}$ has a left adjoint $\textbf{q}: \textbf{RD}_{(c )}\arr \textbf{RD}_{(c,q )}$. Given a regular doctrine $(\C,P)$ we shall denote the free regular doctrine with effective quotients by $(\C_q,P_q)$. $\C_q$ is the category whose objects are pairs $(A,\rho)$ in which $A$ is in $\C$ and $\rho$ is an equivalence relation on $A$. An arrow $f: (A,\rho)\arr(B,\sigma)$ is an arrow $f:A\arr B$ in $\C$ such that $\rho \le (f\times f)^*\sigma$. The functor $P_q$ is determined by the assignment $$P_q(A,\rho) = \{\phi\ \epsilon\ P(A)\ |\ \pi_1^*\phi\wedge\rho\le\pi_2^*\phi\}$$and by $P_q(f)=f^*$. The unite of the adjunction is the family $(\nabla, \zeta)_P:(\C,P)\arr(\C_q,P_q)$ where $\nabla$ is the identity on morphisms and sends an object $A$ of $\C$ to $(A,\delta_A)$ in $\C_q$, while $\zeta$ is the family of identity homomorphisms, since $P_q(A,\delta_A) = P(A)$.\\

As before, the left adjoint restricts to the inclusion of \textbf{LT}$_{(c,q)}$ into \textbf{LT}$_{(c )}$, where the former is the category of triposes with full comprehensions and effective quotients and morphisms are those morphisms of \textbf{LT}$_{(c )}$ which preserve quotients. The proof that the above quotients completion preserves rights adjoints ($\imply$ and $\forall$) is in \cite{RM2}. If $\mathbb{P}(X)$ is a weak power object of $X$ in $\C$, then a power objects of $(X,\rho)$ in $\C_q$ is $$(\mathbb{P}(X), \forall_{<\pi_2,\pi_3>} (<\pi_1,\pi_2>^*\in_X\sse<\pi_1,\pi_3>^*\in_X))$$while$$\in_{(X,\rho)}\ =\ \in_X \wedge\ \forall_{<\pi_2,\pi_3>}(<\pi_1,\pi_2>^*\rho\imply<\pi_1,\pi_3>^*\in_X)$$
Moreover $\nabla\mathbb{P}(X)$ is a strong power object of $\nabla X$. Therefore $(\nabla, \zeta)$ is a morphism in \textbf{LT}$_{c}$. It is straightforward to prove that if $(F,f):(\C,P)\arr (\mathbb{D},R)$ is a logical morphism, then its unique extension $(\overline{F},\overline{f}):(\C_q,P_q)\arr (\mathbb{D},R)$ is logical. Hence we have that the restriction of \textbf{q} to \textbf{LT}$_{(c )}$ fits into the following diagram
\[
\xymatrix{
\textbf{RD}_{(c )}\ar@/^/[r]^-{\textbf{q}}&\textbf{RD}_{(c,q )}\ar@/^/[l]\ar@{}[l]|{\bot}\\
\textbf{LT}_{(c )}\ar@{^(->}[]+<0ex,2.5ex>;[u]\ar@/^/[r]^-{\textbf{q}}&\textbf{LT}_{(c,q )}\ar@{^(->}[]+<0ex,2.5ex>;[u]\ar@/^/[l]\ar@{}[l]|{\bot}
}
\]
\section{Extensionality}\label{sec4}
The notion of extensionality provides a link between the internal equality predicate of a doctrine and the actual equality of terms seen as morphisms of the base category.
\begin{deff}
A regular doctrine $(\C,P)$ is extensional if for every $A$ in $\C$ and every pairs of parallel morphisms $f,g:X\arr A$, it holds that $$f =g\ \ \ \ \text{if and only if}\ \ \ \ \top_X \le <f,g>^*\delta_A$$
\end{deff}
Given a regular doctrine $(\C,P)$, we denote by $(\C_e,P_e)$ the regular doctrine where $\C_e$ has the same objects of $\C$ and morphism $[f]:X\arr A$ are equivalence classes of morphisms of $\C$ with respect to the equivalence relation  $$f \sim g\ \ \ \ \text{if and only if}\ \ \ \ \top_X \le <f,g>^*\delta_A$$ 
the functor $P_e$ is defined by $P_e(A) = P(A)$ while $[f]^*$ is $f^*$. Authors in \cite{RM} proved that reindexing does not depend on the choice of representatives. $(\C_e,P_e)$ is called the extensional collapse of $(\C,P)$.\\

The morphism $(L,\lambda)_P:(\C,P)\arr (\C_e,P_e)$, where $L(f) = [f]$ and $\lambda$ is the family of identity homomorphisms between infsemilattices, constitutes the unite of the adjunction 
\[
\xymatrix{
\textbf{RD}\ar@/^/[r]^-{\textbf{e}}&\textbf{RD}_{(e)}\ar@/^/[l]\ar@{}[l]|{\bot}
}
\]
%
It is immediate to see that the extensional collapse can be restricted to the category of regular doctrines with comprehensions and quotients and also to its subcategory of triposes; in other words we have the following diagram
\[
\xymatrix{
\textbf{RD}_{(c,q )}\ar@/^/[r]^-{\textbf{e}}&\textbf{RD}_{(c,q,e )}\ar@/^/[l]\ar@{}[l]|{\bot}\\
\textbf{LT}_{(c,q )}\ar@{^(->}[]+<0ex,2.5ex>;[u]\ar@/^/[r]^-{\textbf{e}}&\textbf{LT}_{(c,q,e )}\ar@{^(->}[]+<0ex,2.5ex>;[u]\ar@/^/[l]\ar@{}[l]|{\bot}
}
\]

We conclude the section by proving two lemmas, which require the notions of comprehensions, quotients and extensionality seen so far and which turn out to be useful in the proof of proposition \ref{car2}.
\begin{lem}\label{compi}
In an extensional regular doctrine $(\C,P)$ with full comprehensions and quotients, the following hold:
\begin{itemize}
\item[i)] $\C$ has equalizers
\item[ii)] if $f:X\arr Y$ in $\C$ is mono, then $(f\times f)^*\delta_Y = \delta_X$
\item[iii)] if $\rho$ is an equivalence relation on $A$, then the canonical quotients morphism $q:A\arr A/\rho$ is internally surjective, i.e. $\top_{A/\rho}=\exists_q\top_A$
\end{itemize}
\begin{proof}
i) for every pair of parallel arrows $f,g:A\rightrightarrows B$ in $\C$ the following $$\lfloor <f,g>^*\delta_B \rfloor:E\arr A \rightrightarrows B$$ is an equalizer diagram. The composition is equal because of extensionality. The universal property of equalizers comes from the universal property of comprehensions. ii) Suppose $f:X\arr Y$ is mono. By i) $\C$ has pullbacks, thus consider that the pullback of $f$ along itself is isomorphic to $X$. iii) Suppose $\rho$ is an equivalence relation on $A$; consider the canonical quotient map $q:A\arr A/\rho$ and the diagram
\[
\xymatrix{
A\ar[rd]_-{h}\ar[rr]^-{q}&&A/\rho\\
&X\ar[ru]_-{\lfloor\exists_q\top_A\rfloor}&
}
\] 
where $h$ is the unique arrow which comes from the universal property of comprehensions since $\top_A\le q^*\exists_q\top_A$. We have then $$\rho\le(q\times q)^*\delta_{A/\rho} = (h\times h)^*(\lfloor\exists_q\top_A\rfloor\times \lfloor\exists_q\top_A\rfloor)^*\delta_{A/\rho} = (h\times h)^*\delta_X$$ where the last equality holds by ii) since every comprehension arrow is mono. Thus, by the universal property of quotients, we have a unique morphism $k:A/\rho\arr X$ with $k\circ q=h$. Then the two following diagram commutes
\[
\xymatrix{
&A\ar[dr]^-{q}\ar[dl]_-{q}\ar[d]^-{h}&\\
A/\rho\ar[r]_-{k}&X\ar[r]_-{\lfloor\exists_q\top_A\rfloor}&A/\rho
}
\]and, by universal properties of quotients, we have that $\lfloor\exists_q\top_A\rfloor\circ k = id_{A/\rho}$. Thus $\lfloor\exists_q\top_A\rfloor\circ k\circ\lfloor\exists_q\top_A\rfloor = \lfloor\exists_q\top_A\rfloor$, from which $k\circ \lfloor\exists_q\top_A\rfloor = id_X$ sinse comprehensions are mono. Fullness leads to $\top_{A/\rho} = \exists_q\top_A$.
\end{proof}
\end{lem}
\begin{lem}\label{strong}
If $(\C,P)$ is an extensional tripos with comprehensions and effective quotients, then it has strong power objects.
\end{lem}
\begin{proof}
Suppose $\mathbb{P}(A)$ is a weak power object of $A$ in $\C$ and let $\mathcal{P}(A)$ denote the quotient $r:\mathbb{P}(A)\arr \mathcal{P}(A) = \mathbb{P}(A)/\Leftrightarrow_A$
where $$\Leftrightarrow_A = \forall_{<\pi_2,\pi_3>}(<\pi_1,\pi_2>^*\in_A\sse<\pi_1,\pi_3>^*\in_A)$$ For every $B$ and $\phi$ over $A\times B$ we denote by $\chi_\phi$ the morphism $r\circ \{\phi\}$ and by $\text{in}_A$ the following formula over $A\times \mathcal{P}(A)$ $$\text{in}_A =\exists_{<\pi_1,\pi_3>}(<\pi_1,\pi_2>^*\in_A \wedge\ <r\pi_2, \pi_3>^*\delta_{\mathcal{P}(A)})$$ By effectiveness of quotients and internal surjectivity of $r$, we have the following equalities
\begin{equation}\notag
\begin{split}
(id_A\times r\circ\{\phi\})^*\text{in}_A & = \exists_{<\pi_1,\pi_3>}(<\pi_1,\pi_2>^*\in_A \wedge\ <\pi_2, \{\phi\}\circ\pi_3>^*(r\times r)^*\delta_{\mathcal{P}(A)})\\
& = \exists_{<\pi_1,\pi_3>}(<\pi_1,\pi_2>^*\in_A \wedge\ <\pi_2, \{\phi\}\circ\pi_3>^*\Leftrightarrow_A)\\
& = \phi\wedge\exists_{<\pi_1,\pi_3>}<\pi_2, \{\phi\}\circ\pi_3>^*\Leftrightarrow_A\ \ =\ \phi
\end{split}
\end{equation}
\end{proof}
\section{Cauchy-completeness}\label{sec5}
In this section we introduce the definition of cauchy-complete objects in a regular doctrine $(\C,P)$. Recall that, given a regular doctrine $(\C,P)$, a formula $F$ in $P(Y\times A)$ is functional from $Y$ to $A$ if it holds that
\begin{center}
$\top_Y \le \exists_{\pi_2}F$\ \ \ and\ \ \ $<\pi_1,\pi_2>^*F\wedge <\pi_1,\pi_3>^*F\le <\pi_2,\pi_3>^*\delta_A$\end{center}
\begin{deff}\label{cc}
Given a regular doctrine $(\C,P)$, an object $A$ is said cauchy-complete if for every $Y$ and formula $F$ in $P(Y\times A)$ which is functional from $Y$ to $A$
there exists a morphism $f:Y\arr A$ such that $(f\times id_A)^*\delta_A = F$.
\end{deff}
We shall say that a regular doctrine is \textbf{cauchy-complete} if every objects of the base is cauchy-complete.\\\\
\begin{Remark}
The term cauchy-complete is usually introduced in terms of left adjoints. It can be proved that a formula is functional if and only if it is a left adjoint (in the logic of the doctrine). The non trivial part of this is in \cite{ccr}. 
\end{Remark}
\begin{lem}\label{compmono}
In an extensional cauchy-complete regular doctrine $(\C,P)$ with full comprehensions, a morphism $f:A\arr B$ in $\C$ is mono if and only if it is the comprehensions of some formula $\beta$ over $B$.
\end{lem}
\begin{proof} Every comprehension morphism is mono. For the converse, suppose $f:A\arr B$ is mono and consider the formula $\exists_f \top_A$ over $B$. Since $\top_A \le f^*\exists_f\top_A$, by the universal property of comprehensions there exists a morphism $k$ with $\lfloor\exists_f \top_A\rfloor\circ k = f$. Moreover $(\lfloor\exists_f \top_A\rfloor \times f)^*\delta_B$ is functional from the domain of $\lfloor\exists_f \top_A\rfloor$ to $A$, then, by cauchy-completeness, it is the internal graph of a morphism $k'$. By extensionality $k'$ is the unique such a morphism and it is straightforward to see that $k'$ is the inverse of $k$.\end{proof}
After lemma \ref{compmono} we can formulate the following proposition, whose proof can be found in \cite{Jacobs}.
\begin{prop}\label{isosi}
An extensional cauchy-complete regular doctrine $(\C,P)$ with full comprehensions is isomorphic in \textbf{RD} to the subobjects doctrines $(\C,\text{sub})$. 
\end{prop}
Given a regular doctrine $(\C,P)$ we shall denote the free cauchy-complete regular doctrine on $(\C,P)$ by $(\C_l,P_l)$. $\C_l$ has the same objects as $\C$, while morphisms from $A$ to $B$ are functional formulas from $A$ to $B$: identities are internal equalities and the composition is the usual composition of formulas, i.e. if $\phi$ is in $P(A\times B)$ and $\psi$ in $P(B\times C)$, the composition $\psi\circ\phi$ is the formula of $P(A\times C)$ defined by $$\exists_{<\pi_1,\pi_3>} ( <\pi_1,\pi_2>^*\phi\wedge <\pi_2,\pi_3>^*\psi)$$ The action of the functor $P_l$ is defined by the following assignments: $P_l(A) = P(A)$ and $P_l(\phi):P_l(B)\arr P_l(A)$ is the functor which maps every formula $\beta$ in $P_l(B)$ in the following formula of $P_l(A)$ $$\exists_{\pi_1} (\phi\wedge \pi_2^*\beta)$$ Frobenius reciprocity and the fact that $\phi$ is functional ensure that $P_l(\phi)$ is an homomorphism of inf-semilattices. Moreover the assignment $$\exists_\phi \alpha = \exists_{\pi_2} (\phi \wedge \pi_1^*\alpha)$$ produces a left adjoint to $P_l(\phi)$ which satisfies Beck-Chevally condition and Frobenius-Reciprocity.\\

The morphism $(\Gamma,\gamma)_P: (\C,P)\arr (\C_l,P_l)$ where $\Gamma$ is the functor from $\C$ to $\C_l$ which maps every morphism $f:A\arr B$ of $\C$ to the formula $\Gamma f = (f\times id_B)^*\delta_B$ and $\gamma$ is the family of identity homomorphisms of inf-semilattices, constitutes the unite of the following adjunction
\[
\xymatrix{
\textbf{RD}_{(e )}\ar@/^/[r]^-{\textbf{l}}&\textbf{RD}_{(e,l )}\ar@/^/[l]\ar@{}[l]|{\bot}
}
\]
where \textbf{RD}$_{(e,l )}$ is the full subcategory of \textbf{RD}$_{(e )}$ on extensional cauchy-complete regular doctrines.\\\\
In fact, for every morphism $(F,f):(\C,P)\arr (\D,R)$ in \textbf{RD}$_{(e)}$, where $(\D,R)$ is cauchy-complete, define a morphism $(\overline{F},\overline{f}): (\C_l, P_l)\arr (\D,R)$ where $\overline{f} = f$ and the functor $\overline{F}:\C\arr \C_l$ is given by the assignment
$$\xymatrix@R=1.5ex@C=1.8ex{
A\ar[dd]_-\phi&&FA\ar[dd]^-{\phi_F}\\
&\mapsto&\\
B&&FB
}$$ in which $\phi_F$ is the morphism of $\D$ such that its internal graph is equal to $(<F\pi_A,F\pi_B>^{-1})^*f_{A\times B}\phi$. Since $\phi$ is functional from $A$ to $B$ and $f_{A\times B}$ preserves the regular structures, $(<F\pi_A,F\pi_B>^{-1})^*f_{A\times B}\phi$ is functional from $FA$ to $FB$. Cauchy-completeness and extensionality of $(\D,R)$ ensure that the formula is the graph of a unique morphism $\phi_F$.\\\\
Moreover $(\overline{F}, \overline{f})$ uniquely makes the following diagram commute$$
\xymatrix@R=4ex@C=3.6ex{
\C^{op}\ar[rd]^-{P}\ar@/_1pc/[ddr]_{(\Gamma,\gamma)}\ar@/^/[rr]^-{(F,f)}&&\mathbb{D}^{op}\ar[dl]^-{R}\\
&\textbf{ISL}&\\
&\C_l^{op}\ar[u]^-{P_l}\ar@/_1pc/[ruu]_-{(\overline{F},\overline{f})}&
}
$$
This is due to the fact that $\overline{F}\circ \Gamma$ is trivially equal to $F$ on objects and if $h:A\arr B$ is a morphism in $\C$, then $\overline{F}(\Gamma (h))$ is the morphism of $\D$ whose graph is $(<F\pi_A,F\pi_B>^{-1})^*f_{A\times B}(h\times id_B)^*\delta_B$ but $$f_{A\times B}(h\times id_B)^*\delta_B = F(h\times id_{FB})^*f_{B\times B}\delta_B = <F\pi_A,F\pi_B>^*(Fh\times id_{FB})^*\delta_{FB} $$ therefore $(<F\pi_A,F\pi_B>^{-1})^*f_{A\times B}\Gamma(h)$ is the graph of $Fh$ and by extensionality we have $\overline{F}(\Gamma (h)) = Fh$.\\

The following lemma is instrumental to prove that the previous adjunction restricts to categories of richer doctrines.
\begin{lem}
Given a regular doctrine $(\C,P)$
\begin{itemize}
\item[i)] if $(\C,P)$ has full comprehensions, so has $(\C_l,P_l)$
\item[ii)] if $(\C,P)$ has effective quotients, so has $(\C_l,P_l)$
\item[iii)] if $(\C,P)$ is a tripos, so is $(\C_l,P_l)$
\end{itemize}
\end{lem}
\begin{proof}
i) suppose $\alpha$ is a formula in $P_l(A) = P(A)$, since $(\C,P)$ has full comprehensions, there exists the morphism $\lfloor\alpha\rfloor:X\arr A$ in $\C$ with $\lfloor\alpha\rfloor^*\alpha = \top_X$. Then $\Gamma \lfloor\alpha\rfloor$ is a morphism from $X$ to $A$ in $\C_l$ such that $P_l(\Gamma \lfloor\alpha\rfloor)\alpha = \top_X$. Note that, since comprehension are full in $(\C,P)$, we have that $\exists_{\lfloor\alpha\rfloor}\top_X = \alpha$, thus for every formula of $(\C_l,P_l)$, which is functional from $Y$ to $A$ with $P_l(F)\alpha = \top_Y$, $(id_Y\times \lfloor\alpha\rfloor)^*F$ is functional from $Y$ to $X$, moreover the composition $\Gamma \lfloor\alpha\rfloor\circ (id_Y\times \lfloor\alpha\rfloor)^*F$ is equal to $F$. ii) If $q:A\arr A/\rho$ is the quotient in $(\C,P)$ of an equivalence relation $\rho$ over $A$, then $\Gamma q : A\arr A/\rho$ is the quotient of $\rho$ in $(\C_l,P_l)$, in fact, for every functional formula $F$ from $A$ to $B$ with $\rho \le (F\times F)^*\delta_B$, we have that $\xi = \exists_{<\pi_2,\pi_3>}(<\pi_1,\pi_2>^*\Gamma q \wedge <\pi_1,\pi_3>^*F)$ is functional from $A/\rho$ to $B$ with $\xi \circ \Gamma q = F$. iii) Straightforward.
\end{proof}
Note that that $\Gamma \mathbb{P}(A)$ is exactly $\mathbb{P}(A)$ and the same holds for $\delta_{\Gamma \mathbb{P}(A)}$, since $\gamma$ is a family of identitiy morphisms. Moreover a power object of $A$ in $(\C,P)$ is also a power objects of $A$ in $(\C_l,P_l)$, in fact for an object $\phi$ over $A\times B$, the morphism $\Gamma \{\phi\}$ is such that $P_l(id_A\times \Gamma \{\phi\})\in_A = \phi$. A consequence of this is that the unite $(\Gamma, \gamma)$ preserves power objects. Moreover in $(F,f)$ is a logic morphism, its unique extension $(\overline{F}, \overline{f})$ is logic too. Then the left adjoint functor $\textbf{l}: \textbf{RD}_{(e )} \arr \textbf{RD}_{(e,l )}$ restricts to the following commutative diagram
\[
\xymatrix{
\textbf{RD}_{(c,q,e)}\ar@/^/[r]^-{\textbf{l}}&\textbf{RD}_{(c,q,e,l)}\ar@/^/[l]\ar@{}[l]|{\bot}\\
\textbf{LT}_{(c,q,e)}\ar@{^(->}[]+<0ex,2.5ex>;[u]\ar@/^/[r]^-{\textbf{l}}&\textbf{LT}_{(c,q,e,l)}\ar@{^(->}[]+<0ex,2.5ex>;[u]\ar@/^/[l]\ar@{}[l]|{\bot}
}
\]
\section{The Tripos to Topos construction}\label{sec6}
Gluing together the diagrams obtained in the previous sections we have\\
\[
\xymatrix{
\textbf{RD}\ar@/^/[r]^-{\textbf{c}}&\textbf{RD}_{(c )}\ar@/^/[l]\ar@{}[l]|{\bot}\ar@/^/[r]^-{\textbf{q}}&\textbf{RD}_{(c,q )}\ar@/^/[l]\ar@{}[l]|{\bot}\ar@/^/[r]^-{\textbf{e}}&\textbf{RD}_{(c,q,e )}\ar@/^/[l]\ar@{}[l]|{\bot}\ar@/^/[r]^-{\textbf{l}}&\textbf{RD}_{(c,q,e,l)}\ar@/^/[l]\ar@{}[l]|{\bot}\\
\textbf{LT}\ar@{^(->}[]+<0ex,2.5ex>;[u]\ar@/^/[r]^-{\textbf{c}}&\textbf{LT}_{(c )}\ar@{^(->}[]+<0ex,2.5ex>;[u]\ar@/^/[l]\ar@{}[l]|{\bot}\ar@/^/[r]^-{\textbf{q}}&\textbf{LT}_{(c,q )}\ar@{^(->}[]+<0ex,2.5ex>;[u]\ar@/^/[l]\ar@{}[l]|{\bot}\ar@/^/[r]^-{\textbf{e}}&\textbf{LT}_{(c,q,e )}\ar@{^(->}[]+<0ex,2.5ex>;[u]\ar@/^/[l]\ar@{}[l]|{\bot}\ar@/^/[r]^-{\textbf{l}}&\textbf{LT}_{(c,q,e,l)}\ar@{^(->}[]+<0ex,2.5ex>;[u]\ar@/^/[l]\ar@{}[l]|{\bot}
}
\]
\\It has been remarked in \cite{RM2} that the action on objects of the functor from \textbf{LT} to \textbf{RD}$_{(c,q,e,l)}$ is, up to equivalence, the tripos to topos construction of Hyland-Johnstone and Pitts \cite{tripinret}.
In this section we show that the category \textbf{LT}$_{(c,q,e,l)}$ is equivalent to $\mathcal{L}\bf{Topos}$, the category of elementary toposes and logic functors. 
\begin{prop}\label{car2}
A non-empty category $\C$ with binary products is an elementary topos if and only if there exists an cauchy-complete extensional tripos $(\C,P)$ with full comprehensions and effective quotients.
\end{prop}
\begin{proof}
Suppose $\C$ is an elementary topos, then consider its subobjects tripos $(\C,\text{sub})$. In a suobjects doctrine every formula is its own full comprehension. Effective quotients comes from exactness of $\C$. Every functional formula is the graph of a unique morphism \cite{Jacobs}, then $(\C,\text{sub})$ is also cauchy-complete and extentional.\\\\
Conversely, if $(\C,P)$ is an object of \textbf{LT}$_{(c,q,e,l)}$, then $\C$ has\\
\\
(finite limits): $\C$ has binary products and equalizers by  lemma \ref{compi}. Since $\C$ is not empty, it has an object $A$. Consider the quotient $$A\arr A/\top_{A\times A}$$for every $Y$ in $\C$ the the top element of $P(Y\times A/\top_{A\times A})$ is functional from $Y$ to $A/\top_{A\times A}$ then by cauchy-completeness it is the internal graph of a morphism $Y\arr A/\top_{A\times A}$ in $\C$. Its uniqueness comes from extentionality.\\
\\
(power objects): by lemma \ref{strong} we know that $(\C,P)$ has strong power objects. Using the notation of \ref{strong}, we denote by $\mathcal{P}(A)$ the strong power objects of $A$. Given a formula $\phi$ over $A\times B$, the following diagram is a pullback 
\[
\xymatrix{
X\ar[rr]\ar[d]_-{\lfloor\phi\rfloor}&&\epsilon_A\ar[d]^-{\lfloor\text{in}_A\rfloor}\\
A\times B\ar[rr]_-{id_A\times \chi_\phi}&&A\times \mathcal{P}(A)
}
\]
where $\epsilon_A$ denotes the domain of the comprehension of $\text{in}_A$. By extentionality and fullness of comprehensions, $\chi_\phi$ is the unique such function. Apply lemma \ref{compmono} to see that $\mathcal{P}(A)$ is a power object of $A$ also in the topos theoretic sense.
\end{proof}
Thus the functor from {$\mathcal{L}\bf{Topos}$ to \textbf{LT}$_{(c,q,e,l)}$ determined by the assigment 
$$\xymatrix@R=1.5ex@C=1.8ex{
\C\ar[dd]_-{F}&&(\C,\text{sub})\ar[dd]^-{(F,F)}\\
&\mapsto&\\
\D&&(\D,\text{sub})
}$$
is an equivalence. It is clearly faithful. It is also full, since if $(G,g):(\C, \text{sub}) \arr (\D, \text{sub})$ is a morphism in \textbf{LT}$_{(c,q,e,l)}$ and $\xi$ a formula over $A$, then the comprehension of $G\xi$ is $G\xi$ itself and since $(G,g)$ preserves comprehensions we have that $G\xi$ is isomorphic to the comprehension of $g(\xi)$, which in turn is $g(\xi)$ itself, then $(G,g)$ is $(G,G)$. Essential surjectiveness comes from \ref{isosi} and  \ref{car2}.
\begin{Remark}
Denoting by \textbf{RT} the full subcategory of \textbf{RD} on triposes whose bases have finite products, it has been proved in \cite{frey} the existence of the following adjointness situations 
\[
\xymatrix{
\textbf{RT}\ar@/^/[r]^-{\mathcal{F}_1}&\textbf{Q\ }\ar@/^/[l]\ar@{}[l]|{\bot}\ar@/^/[r]^-{\ \ \mathcal{F}_2}&\mathcal{R}\textbf{Topos}\ar@/^/[l]\ar@{}[l]|{\bot\ \ }
}
\]
where \textbf{Q} is the category of q-toposes and regular functors, i.e. functors preserving left adjoints (definitions are in \cite{frey}) and $\mathcal{R}$\textbf{Topos} is the full subcategory on toposes of \textbf{Reg} (the category of regular categories and regular functors). It can be proved, by arguments similar to those in \ref{car2}, that \textbf{Q} is equivalent to \textbf{RT}$_{(c,q,e)}$ and that \textbf{RT}$_{(c,q,e,l)}$ is equivalent to $\mathcal{R}$\textbf{Topos}. Thus (up to equivalences) $\mathcal{F}_1$ is the restriction of \textbf{eqc} to \textbf{RT} and $\mathcal{F}_2$ is the restriction of \textbf{l} to \textbf{RT}$_{(c,q,e)}$.
\end{Remark}
\begin{Remark}
As we pointed out in lemma \ref{strong}, a tripos in \textbf{LT}$_{(c,q,e,l)}$ has strong power objects and these are basically obtained as quotients of weak power objects with respect to suitable equivalence relations. A sort of converse lemma can also be proved: an extensional cauchy-complete tripos $(\C,P)$ with full comprehensions and strong power objects has effective quotients. Given an object $A$ and an equivalence relation $\rho$ on $A$, define $A/\rho$ to be the domain of the following formula over $\mathbb{P}A$
$$\sigma = \exists_{\pi_1}( \in_A\wedge\ \forall_{<\pi_1,\pi_3>}(<\pi_1,\pi_2>^*\rho \leftrightarrow <\pi_2,\pi_3>^*\in_A))$$If $f:A\arr Y$ is such that $\rho\le(f\times f)^*\delta_Y$ then$$\exists_{<\pi_2,\pi_3>}(<\pi_1,\pi_2>^*(id_A\times \lfloor\sigma\rfloor)^*\in_A\wedge\ <\pi_1,\pi_2>^*(id_A\times f)^*\delta_Y)$$ is functional from $A/\rho$ to $Y$, and by extensionality and cauchy-completeness, it is the graph of the unique desired morphism.\\\\
Denote by \textbf{LT$_s$} the subcategory on \textbf{LT} of triposes with strong power objects, then we have the equivalence $\textbf{LT$_{s(c,e,l)}$} \equiv \textbf{LT}_{(c,q,e,l)}$. It has been shown by Carboni in \cite{carbons} the existence of the following adjoint situation
\[
\xymatrix{
\textbf{LT$_s$}\ar@/^/[r]^-{\mathcal{F}}&\mathcal{L}\textbf{Topos}\ar@/^/[l]\ar@{}[l]|{\bot\ \ }
}
\]
It is immediate to prove that the reflector $\mathcal{F}$ is the restriction of \textbf{lec} to \textbf{LT$_s$}.\end{Remark}

\bibliography{biblio}
\bibliographystyle{plain}

\end{document}